\documentclass{birkmult}
\usepackage{amsmath,latexsym,amssymb,amsfonts,amsbsy,graphicx,mathrsfs}
\usepackage[active]{srcltx}
\usepackage[colorlinks,linkcolor={blue},
citecolor={blue},urlcolor={red},]{hyperref}

 \newtheorem{theorem}{Theorem}[section]
 
 \newtheorem{lemma}[theorem]{Lemma}
 \newtheorem{proposition}[theorem]{Proposition}
 \theoremstyle{definition}
 \newtheorem{definition}[theorem]{Definition}
 \theoremstyle{remark}
 \newtheorem{remark}[theorem]{Remark}
 
 \numberwithin{equation}{section}

\let\mathcal \undefined
\def\mathcal{\mathscr}

\usepackage{newsymbol}
\let\emptyset \undefined
\let\ge       \undefined
\let\le       \undefined
\newsymbol\le          1336  
\newsymbol\ge          133E  \let\geq\ge
\newsymbol\emptyset    203F
\newsymbol\notle       230A
\newsymbol\notge       230B

\newcommand{\g}{\gamma}
\renewcommand{\d}{\delta}
\newcommand{\e}{\varepsilon}
\newcommand{\R}{{\mathbb R}}
\newcommand{\cd}{\cdot}

\newcommand{\mE}{\mathbb{E}}

\newcommand{\dbR}{{\mathop{\rm l\negthinspace R}}}
\newcommand{\E}{{\mathbb E}}
\newcommand{\calL}{{\mathscr L}}
\renewcommand{\P}{{\mathbb P}}
\newcommand{\F}{{\mathbb F}}
\newcommand{\calF}{{\mathscr F}}
\newcommand{\n}{\Vert}
\newcommand{\one}{{{\bf 1}}}
\newcommand{\s}{^*}
\newcommand{\lb}{\langle}
\newcommand{\rb}{\rangle}

\newcommand{\iprod}[2]{(#1|#2)}
\newcommand{\ga}{\gamma}
\newcommand{\ggX}{\ga(0,T;X)}
\newcommand{\ud}{\,{\rm d}}
\newcommand{\om}{\omega}
\newcommand{\Om}{\Omega}

\newcommand{\wt}{\widetilde}
\newcommand{\Dom}{\mathsf{D}}
\allowdisplaybreaks

\begin{document}

 \title[BSEEs in UMD Banach spaces]{Backward stochastic evolution equations\\ in UMD Banach spaces}
 
 \author{Qi L\"u}
\address{School of Mathematics, Sichuan University, Chengdu 610064, China}
 \email{lu@scu.edu.cn}

 \author{Jan van Neerven}
 \address{Delft University of Technology, P.O. Box 5031, 2600 GA Delft, The Netherlands}
 \email{J.M.A.M.vanNeerven@TUDelft.nl}

 \date{\today}
 \thanks{The first-named author is supported by the NSF of China under grant 11471231 and Grant MTM2014-52347 of the MICINN, Spain. 
This paper was started while the second-named author visited Sichuan University. He would like to thank the School of Mathematics for its kind hospitality. }

\begin{abstract} 
Extending results of Pardoux and Peng \cite{PP} and Hu and Peng \cite{HP}, we prove well-posedness results
for backward stochastic evolution equations in UMD Banach spaces.
\end{abstract}

\dedicatory{Dedicated to Ben de Pagter on the occasion of his 65th birthday}

\date{\today}

\maketitle

\section{Introduction}

In this paper we extend the classical results of Pardoux and Peng \cite{PP} and Hu and Peng \cite{HP} on backward stochastic differential equations to the UMD-valued setting.

We consider backward stochastic evolution equations (BSEEs) of the form
\begin{equation}\label{eq:BSDE-intr}\tag{BSEE}
\begin{cases}
 \ud U(t) + AU(t)\ud t = f(t,U(t),V(t))\ud t + V(t)\ud W(t), \quad t\in [0,T], \\
  U(T) = u_T,\end{cases}
\end{equation}
where $-A$ is the generator of a $C_0$-semigroup $S=(S(t))_{t\ge 0}$ on a UMD Banach space $X$ and $W= (W(t))_{t\in [0,T]}$ is a standard Brownian motion. 
Our results extend to finite-dimensional Brownian motions and, more generally, to cylindrical Brownian motions without
difficulty, but we do not pursue this here in order to keep the presentation as simple as possible.
Denoting by $\F = \{\calF_t\}_{t\in [0,T]}$ the augmented filtration 
generated by the Brownian motion $W$, the final value $u_T$ is taken from 
$L^p(\Om,\calF_T;X)$, the closed subspace $L^p(\Om;X)$ of all functions having a
strongly $\calF_T$-measurable pointwise defined representative. The mapping $f$ is assumed to be 
$\F$-adapted and to satisfy suitable integrability and Lipschitz continuity requirements with respect to the natural norm arising from the $L^p$-stochastic integral in $X$. We will be interested in 
$L^p$-solutions $(U,V)$ with values in $X$.

BSEEs, as infinite dimensional extensions of backward
stochastic differential equations,
arise in many applications related to stochastic
control.
For instance, the
Duncan--Mortensen--Zakai filtration equation for
the optimal control problem of partially
observed stochastic differential equations is a
linear BSEE (see, e.g., \cite{Bensoussan}); in order to establish the maximum principle for the
  optimal control problem of stochastic evolution equations one needs to introduce a linear BSEE  as the adjoint
equation (see, e.g., \cite{LZ2,Zhou1}); in the study of controlled non-Markovian SDEs
 the stochastic Hamilton--Jacobi--Bellman equation is a class of fully nonlinear BSEEs (see, e.g., \cite {EK,Peng}); 
 and when the coefficients of the stochastic differential equation describing the stock price are  random
processes, the stochastic version of the
 Black-Scholes formula for option pricing is a
 BSEE (see, e.g., \cite{MaYong}).

In a Hilbert space setting, BSEEs have already be studied in
\cite{HP}; see also \cite{AnhGreYon, AnhYon, FuhTes, LZ1, LZ2, LZ3} and the references cited therein. 
In \cite{DT,MaYong,MaYong1}
the existence of a solution
in the Sobolev space $W^{m,2}$ is obtained, in \cite {Azimi, DQT} the
existence of a solution in $L^q$, and in \cite{TW} the
existence of a solution in H\"older spaces.

In the present paper, we study BSEEs in the abstract
framework of evolution equations on UMD Banach spaces. The main results in
\cite{DQT,DT,MaYong,MaYong1} are covered by
our results. Furthermore, our results
can be used to show the well-posedness of many
other backward stochastic partial differential
equations, such as $2m$-order backward
stochastic parabolic equations.

The second-named author would like to use this opportunity to express warm-felt gratitude to Ben 
for invaluable mentorship and support throughout an entire mathematical career.
Thanks for all, Ben!

\section{Preliminaries}\label{sec:preliminaries}

In this section we recall some useful concepts
and results which will be used in the course of the paper. 
Proofs and more details, as well as references
to the literature, can be
found in the papers \cite{Burk, KW, Neerven, NVW15b},
the lecture notes \cite{DHP, KuWe}, and 
the monographs \cite{HNVW1, HNVW2, Pisier}.

Unless stated otherwise, all vector spaces are assumed to be real.
We will always identify Hilbert spaces with their duals by
means of the Riesz representation theorem.

\subsection{$\gamma$-Boundedness}

Let $X$ and $Y$ be Banach spaces and let
$\{\g_n\}_{n\ge 1}$ be Gaussian sequence (i.e., a
sequence of independent real-valued standard
Gaussian random variables).

\begin{definition}
A family $\mathscr{T}$ of bounded linear
operators from $X$ to $Y$ is called {\em
$\g$-bounded} if there exists a constant $C\ge
0$ such that for all finite sequences
$\{x_n\}_{n=1}^N$ in $X$ and $\{T_n\}_{n=1}^N$ in
${\mathscr {T}}$ we have
$$
\mE \Big\n \sum_{n=1}^N \g_n T_n x_n\Big\n^2 \le
C^2\mE \Big\n \sum_{n=1}^N \g_n x_n\Big\n^2.
$$
\end{definition}

Clearly, every
$\g$-bounded family of bounded linear operators
from $X$ to $Y$ is uniformly bounded and
$\sup_{t\in \mathscr{T}} \n T\n_{\calL(X;Y)} \le
C$, the constant appearing in the above definition. 
In the setting of Hilbert spaces
both notions are equivalent and the above inequality holds with
$C=\sup_{t\in \mathscr{T}} \n T\n_{\calL(X;Y)}$.

$\gamma$-Boundedness is the Gaussian
analogue of $R$-boundedness, obtained by replacing Gaussian variables
by Rademacher variables. This notion was introduced and thoroughly studied
in the seminal paper \cite{CPSW}.

\subsection{$\gamma$-Radonifying operators}\label{subsec:radonif}

Let $H$ be a Hilbert space with inner product
$\iprod{\cdot}{\cdot}$ and $X$ a Banach
space. Let $H\otimes X$ denote the linear space
of all finite rank operators from $H$ to $X$.
Every element in $H \otimes X$ can be
represented in the form $\sum_{n=1}^N h_n\otimes
x_n$, where $h_n\otimes x_n$ is the rank one
operator mapping the vector $h\in H$ to
$\iprod{h}{h_n}x_n \in X$. By a Gram-Schmidt
orthogonalisation argument we may always assume that
the sequence  $\{h_n\}_{n=1}^N$ is orthonormal in
$H$.

\begin{definition}\label{def-g-rad}
The Banach space $\g( H ,X)$ is the
completion of $ H \otimes X$ with respect to the
norm
\begin{equation}\label{def-g-rad-eq1}
\Big\n \sum_{n=1}^N h_n\otimes
x_n\Big\n_{\g(H,X)} :=  \Big(\mE
\Big\n \sum_{n=1}^N \g_n x_n\Big\n^2\Big)^{1/2},
\end{equation}
where  $\{h_n\}_{n=1}^N$ is orthonormal in $H$
and $\{\gamma_n\}_{n= 1}^N$ is a Gaussian
sequence.
\end{definition}

Since the distribution of a Gaussian vector in
$\dbR^N$ is invariant under orthogonal
transformations, the quantity on the right-hand
side of \eqref{def-g-rad-eq1} is independent of
the representation of the operator as a finite sum of the 
form $\sum_{n=1}^N h_n\otimes x_n$ as long as
$\{h_n\}_{n=1}^N$ is orthonormal in $H$.
Therefore, the norm $\n\cdot\n_{\gamma(H,X)}$ is
well defined.

\begin{remark}
 By the Kahane-Khintchine inequalities \cite[Theorem 6.2.6]{HNVW2}, for all $0<p<\infty$ there exists
 a universal constant $\kappa_p$, depending only on $p$, such that for all
 Banach spaces $X$ and all finite sequences $\{x_n\}_{n=1}^N$ in $X$ we have
 $$ \frac1{\kappa_p} \Big(\mE
\Big\n \sum_{n=1}^N \g_n x_n\Big\n^p\Big)^{1/p}\le \Big(\mE
\Big\n \sum_{n=1}^N \g_n x_n\Big\n^2\Big)^{1/2}\le \kappa_p\Big(\mE
\Big\n \sum_{n=1}^N \g_n x_n\Big\n^p\Big)^{1/p}.
$$
As a consequence, for $1\le p<\infty$ the norm
\begin{equation*}
\Big\n \sum_{n=1}^N h_n\otimes
x_n\Big\n_{\g^p(H,X)} :=  \Big(\mE
\Big\n \sum_{n=1}^N \g_n x_n\Big\n^p\Big)^{1/p},
\end{equation*}
with $\{h_n\}_{n=1}^N$ orthonormal in $H$, is an equivalent norm on $\gamma(H,X)$.
Endowed with this equivalent norm, the space is denoted by $\gamma^p(H,X)$.
\end{remark}

For any Hilbert space $H$ we have a natural
isometric isomorphism $$\gamma(H,X)
=\calL_2(H,X),$$ where $\calL_2(H,X)$ is the space
of all Hilbert-Schmidt operators from $H$ to
$X$. Furthermore, for $1\le p<\infty$ and $\sigma$-finite measures $\mu$ we have an isometric
isomorphism of Banach spaces
\begin{equation}\label{eq-gammaLpX}
\gamma^p(H, L^p(\mu;X)) \simeq
L^p(\mu;\gamma^p(H;X))
\end{equation} which is
obtained by associating with $f\in
L^p(\mu;\gamma(H;X))$ the mapping $h'\mapsto
f(\cdot)h'$ from $H$ to $L^p(\mu;X)$ \cite[Theorem 9.4.8]{HNVW2}.  In
particular, upon identifying $\gamma(H,\R)$
 with $H$, we obtain an
isomorphism of Banach spaces
\begin{equation*}
\gamma(H, L^p(\mu)) \simeq L^p(\mu;H).
\end{equation*}

When $I$ is an interval in the real line, for brevity we write $$\gamma(I;X) := \gamma(L^2(I),X).$$ 

\begin{definition}\label{def:belong}
A strongly measurable function $f:I\to X$ is said to {\em define 
an element of $\gamma(I;X)$}
if $\lb f,x\s\rb \in L^2(I)$ for all $x\s\in X\s$ and the Pettis integral operator
$$g \mapsto \int_I f(t)g(t)\ud t$$
belongs to $\g(I;X)$. 
\end{definition}

Observe that the condition $\lb f,x\s\rb \in L^2(I)$ for all $x\s\in X\s$ ensures
that $fg$ is Pettis integrable for all $g\in L^2(I)$; see \cite[Definition 9.2.3]{HNVW2} and the discussion
following it.

Throughout the paper we fix a final time $0<T<\infty$. For any $f\in \ggX$ it is possible to define a $\frac12$-H\"older continuous function 
$[0,T]\ni t\mapsto \int_0^t f(s)\ud s \in X$ as follows.
We begin by observing that integration operator $I_{s,t}:\phi\mapsto \int_s^t f(r)\ud r$ is bounded from $L^2(0,T)$ to $\R$
and has norm $(t-s)^{1/2}$. Therefore, by the Kalton--Weis extension theorem (\cite[Theorem 9.6.1]{HNVW2})
the mapping $\wt I_{s,t}: \phi\otimes x \mapsto (I_{s,t}\phi)\otimes x$ has a 
unique extension to a bounded linear operator from
$\gamma(0,T;X)$ to $X$ of the same norm: $\n \wt I_{s,t}\n_{\calL(\gamma(0,T;X),X)} = \n I_{s,t}\n_{\calL(L^2(0,T),\R)} = (t-s)^{1/2}.$
We now define, for $g\in \ggX$,
$$ \int_s^t f(s)\ud s := \wt I_{s,t} f .$$
Noting that $\wt I_{0,t} f - \wt I_{0,s} f = \wt I_{s,t} f$, we see that $t\mapsto \int_0^t f(s)\ud s$ is H\"older continuous of order $\frac12$ and
\begin{align}\label{eq:est-t} \Big\n \int_s^t  f(s)\ud s\Big\n \le (t-s)^{1/2} \n f\n_{\ggX}.
\end{align}

\begin{remark}\label{rem:abuse}
We are abusing notation slightly here, as the above integral notation is only formal since elements in $\ggX$ cannot in general be represented as functions. 
For the sake of readability this notation will be used throughout the paper.
\end{remark}
Treating $t$ as a variable, we may also use the Kalton--Weis extension theorem
to extend $f\mapsto \int_0^{\cdot} f(s)\ud s$ (viewed as a bounded operator on $L^2(0,T)$
of norm $T/\sqrt{2}$)
to a bounded operator on $\ggX$ of the same norm. With the same slight abuse of notation this may be expressed as
\begin{align*}
\Big\n t\mapsto \int_0^t  f(s)\ud s\Big\n_{\ggX} \le \frac{T}{\sqrt{2}} \n f\n_{\ggX}.
\end{align*}

We will need the following elaboration on this theme, which is of some independent interest.
Put $$\Delta := \{(s,t)\in (0,T)\times (0,T):\, 0< s \le t< T\}.$$

\begin{lemma}\label{lem:gammaDeterm}
Let $X$ and $Y$ be Banach spaces and 
assume that $Y$ does not contain a closed subspaces isomorphic to $c_0$.
\begin{enumerate}
 \item[\rm(1)] Let $M:(0,T)\to \calL(X,Y)$ be a function with the property that $t\mapsto M(t)x$ 
is strongly measurable for all $x\in X$ and assume that $M$ has $\gamma$-bounded range, with $\gamma$-bound $\gamma(M)$.
Then the function
$$ \Phi f: t\mapsto \int_0^t M(t-s)f(s)\ud s, \quad f\in L^2(0,T)\otimes X,$$
defines an element of $\gamma(0,T;Y)$ of norm
$$ \n \Phi f\n_{\gamma(0,T;Y)} \le T \gamma(M)\n f\n_{\gamma(0,T;X)}.$$
\item[\rm(2)]  Let $M:\Delta\to \calL(X,Y)$ be a function with the property that $(s,t)\mapsto M(s,t)x$ 
is strongly measurable for all $x\in X$ and assume that $M$ has $\gamma$-bounded range, with $\gamma$-bound $\gamma(M)$. The function
$$ \Phi f: t\mapsto \int_0^t M(s,t)f(s,t)\ud s, \quad f\in L^2(\Delta)\otimes X,$$
defines an element of $\gamma(0,T;Y)$ of norm
$$ \n \Phi f\n_{\gamma(0,T;Y)} \le T^{1/2} \gamma(M)\n f\n_{\gamma(\Delta;X)}.$$
\end{enumerate}
As a consequence, the mappings $f\mapsto \Phi f$ extend uniquely to bounded operators from
$\gamma(0,T;X)$ to $\gamma(0,T;Y)$ and from $\gamma(\Delta;X)$ to $\gamma(0,T;Y)$, respectively,  of norms at most $T\gamma(M)$ and $T^{1/2}\gamma(M)$, respectively. 
\end{lemma}

\begin{proof} We begin with the proof of (1). 
The estimate
$$ \int_0^T \int_0^t |g(t-s)|^2\ud s\ud t \le T \n g\n_2^2$$
shows that the
mapping $J_1: g \mapsto [(s,t)\mapsto g(t-s)]$ is bounded from $L^2(0,T)$ to $L^2(\Delta_T)$ of norm at most $T^{1/2}.$
By the Kalton--Weis extension theorem, it extends to a bounded operator from
$\gamma(0,T;X)$ to $\gamma(\Delta;X)$ of the same norm. By the Kalton--Weis multiplier theorem (\cite[Theorem 9.5.1]{HNVW2}),
the pointwise multiplier $M$ (acting in the variable $s$, so that $[(s,t)\mapsto g(t-s)]$ is mapped to 
$[(s,t)\mapsto M(s)g(t-s)]$) extends to a bounded operator from $\gamma(\Delta;X)$
to $\gamma(\Delta;Y)$ of norm at most $\gamma(M)$. 
Next, the estimate 
$$\int_0^T \Big|\int_0^t f(s,t)\ud s\Big|^2\ud t
\le T\int_0^T \int_0^t |h(s,t)|^2\ud s\ud t
$$
shows that the mapping $J_2: h\mapsto [t\mapsto \int_0^t h(s,t)\ud s]$ is bounded from $L^2(\Delta_T)$
to $L^2(0,T)$ of norm at most $T^{1/2}$. 
By the Kalton--Weis extension theorem, it extends to a bounded operator from
$\gamma(\Delta;Y)$ to $\gamma(0,T;Y)$ of the same norm.
The mapping $f\mapsto \Phi f$ in the statement of the lemma factorises as $\Phi = J_2\circ M \circ J_1$ and therefore extends 
to a bounded operator from $\gamma(0,T;X)$ to $\gamma(0,T;Y)$ of norm at most $T\gamma(M)$. 

\smallskip
(2): \ This is proved similarly, except that the first step of the proof can now be skipped.
\end{proof}

\subsection{UMD spaces and the upper contraction property}

We next introduce the class of Banach spaces in which we 
will be working.

\begin{definition}\label{def-UMD}
A Banach space $X$ is called a {\em UMD space}
if for some (equivalently,
for all) $1<p<\infty$ there is a constant
$C_{p,X}\ge 0$ such that for all finite $X$-valued
$L^p$-martingale difference sequences
$\{d_n\}_{n=1}^N$ on a probability space $\Omega$ and sequences of signs
$\{\epsilon_n\}_{n=1}^N$ one has
\begin{equation*}
\mE \Big\n \sum_{n=1}^N \e_n d_n\Big\n^p \le C_{p,X}^p
\mE \Big\n\sum_{n=1}^N d_n\Big\n^p, \quad \forall
N\geq 1.
\end{equation*}
\end{definition}

Every Hilbert space and every space $L^p(\mu)$ with $1<p<\infty$ is a UMD space. 
If $X$ is a UMD space, then the spaces
$L^p(\mu;X)$ are UMD for all $1<p<\infty$.  Moreover,
$X$ is a UMD space if and only $X\s$ is a UMD space.
Every UMD space is reflexive (and in fact super-reflexive); it follows that spaces such as
$c_0$, $C(K)$, $\ell^\infty$, $L^\infty(\mu)$, $\ell^1$, $L^1(\mu)$, and all Banach spaces containing isomorphic copies of one of these spaces, fail the UMD property (apart from the trivial cases giving rise to finite-dimensional spaces, i.e., when $K$ is finite or $\mu$ is supported on finitely many atoms).

\begin{definition}\label{def:triangular}
A Banach space $X$ has the {\em upper contraction property} if for some (equivalently, for all) $1\le p<\infty$ 
there is a constant $ C_{p,X}\ge 0$ such that for
all finite sequences  $\{x_{mn}\}_{m,n=1}^{M,N}$ in $X$ and all Gaussian sequences  $\{\gamma_m'\}_{m=1}^M$ and
$\{\gamma_n''\}_{n=1}^N$ on independent probability spaces $\Om'$ and $\Om''$ and 
$\{\gamma_{m,n}\}_{m,n=1}^{M,N}$ on a probability space $\Om$,
we have
\[ \E\Bigl\| \sum_{m=1}^M \sum_{n=1}^N \gamma_{mn} x_{mn} \Bigr\|^p \le C_{p,X}^p
 \E'\E''\Bigl\| \sum_{m=1}^M \sum_{n=1}^N \gamma_m' \gamma_n'' x_{mn}
\Bigr\|^p .\]
\end{definition}
By interchanging the two double sums one obtains the related {\em lower contraction property},
and a Banach space is said to have the {\em Pisier contraction property} if it has both the upper and lower
contraction property. In the present paper we only need the upper contraction property. 

Every Hilbert space and every Banach lattice with finite cotype
(in particular, every space $L^p(\mu)$ with $1\le p<\infty$) 
has the Pisier contraction property.  If $X$ has the upper (resp. lower, Pisier) contraction property, then the spaces
$L^p(\mu;X)$ have the upper (resp. lower, Pisier) contraction property for all $1\le p<\infty$. Moreover, if $X$ is $K$-convex, then
$X$ has the upper (resp. lower, Pisier) contraction property if and only $X\s$ has the lower (resp. upper, Pisier) contraction property. Every Banach space with type $2$ has the upper contraction property. The reader is referred to \cite[Section 7.6]{HNVW2} for proofs and more details. 

The following lemma translates the above definition into the language of $\gamma$-radonification. A proof is 
obtained by noting that for functions in $L^2(0,T)\otimes L^2(0,T)\otimes X$ the lemma follows from the 
estimate of the definition, and the general case follows from it by approximation.

\begin{lemma}\label{lem:triangular} 
If $X$ is a Banach space with the upper contraction property, then for all $f\in L^2(0,T)\otimes L^2(0,T)\otimes X$ we have
$$ \n f\n_{\gamma((0,T)\times(0,T);X)} \le C_{p,X} \n f\n_{\gamma(0,T; \gamma(0,T;X))}.$$
\end{lemma}

\subsection{Stochastic integration}

Let $\F= (\calF_t)_{t\in [0,T]}$ be a filtration in $\Omega$. An $X$-valued {\em $\calF$-adapted step process}
is a finite linear combination of indicator processes of the form $\one_{(s,t)\times F}\otimes x$
with $F\in \calF_s$ and $x\in X$.
The space $$L_\F^p(\Om;\ggX)$$ is defined as the closure  in $L^p(\Om;\ggX)$ of the
$X$-valued $\calF$-adapted step processes.
The following result is from \cite{NVW07}.

\begin{lemma}\label{lem:adapted}
 If the process $\phi:[0,T]\times \Om\to X$ is $\F$-adapted and defines an element of
 $L^p(\Om;\ggX)$, then it defines an element of $L_\F^p(\Om;\ggX)$.
\end{lemma}

From the point of view of stochastic integration, the raison d'\^etre for UMD spaces is the following result
of \cite{NVW07}.
\begin{theorem}[It\^o isomorphism]\label{thm:Itoisomorph}
Let $X$ be a UMD space and let $1<p<\infty$. For all $\F$-adapted elementary processes
$\phi\in L^p(\Omega;\gamma(0,T;X))$
we have
$$
 \E \Big\n \int_0^T \phi\ud W \Big\n^p 
  \eqsim_{p}
 \E \sup_{t\in [0,T]}\Big\n \int_0^t \phi\ud W \Big\n^p
  \eqsim_{p,X}\n\phi\n_{L^p(\Omega;\gamma(0,T;X))}^p
$$
with implied constants depending only on $p$ and $X$.
\end{theorem}

As an immediate consequence, the stochastic integral can be extended to arbitrary integrands in
$L_\F^p(\Om;\ggX)$, with the same two-sided bound on their $L^p$-moments. 
It can furthermore be shown (see \cite{Garling}) that the UMD property is necessary in Theorem \ref{thm:Itoisomorph} in the sense that it is
implied by the validity of the statement in the theorem.

\begin{remark}
For $\phi\in L_\F^p(\Om;\ggX)$
we denote by $\int_0^T \phi\ud W$ the unique extension of the stochastic integral
as guaranteed by the theorem. For $t\in [0,T]$ we write $\int_0^t \phi\ud W:=\int_0^T \one_{(0,t)}\phi\ud W$.
\end{remark}

\section{Backward stochastic evolution equations: well-posedness}

Let us now take up our main topic, the study of the backward stochastic evolution equation (BSEE)
\begin{equation}\label{eq:BSDE}\tag{BSEE}
\begin{cases}
 \ud U(t) + AU(t)\ud t = f(t,U(t),V(t))\ud t + V(t)\ud W(t), \quad t\in [0,T], \\
  U(T) = u_T.\end{cases}
\end{equation}
The function $f$ also depends on the underlying probability space, but following common practice we suppress this from the notation.
The following standing assumptions, or, when this is explicitly indicated, a selection of them, will be in force throughout the remainder of the paper:\smallskip

\begin{itemize}
 \item [\rm(H1)] $X$ is a UMD Banach space and $1<p<\infty$; \smallskip
 \item [\rm(H2)] $\F = \{\calF_t\}_{t\in [0,T]}$ is the augmented filtration generated by the Brownian motion $W = (W(t))_{t\in [0,T]}$;\smallskip
 \item [\rm(H3)] $u_T$ belongs to  $L^p(\Om,\calF_T;X)$;\smallskip
 \item [\rm(H4)] $A$ generates a $C_0$-semigroup $S = \{S(t)\}_{t\ge 0}$ on $X$;\smallskip
 \item [\rm(H5)] the set $\{S(t)\}_{t\in [0,T]}$ is $\gamma$-bounded.\smallskip
\end{itemize}
If $X$ is isomorphic to a Hilbert space, (H5) follows from (H4).
If $X$ is a UMD space, (H4) and (H5) are fulfilled when $A$ has maximal $L^p$-regularity on $[0,T]$.
Recall that a densely defined, closed operator $A$ acting in a Banach space $X$ has {\em maximal $L^p$-regularity on $[0,T]$} if
there exist a constant $C\ge 0$ such that for every $f\in C_{\rm c}(0,T)\otimes \Dom(A)$ there exists
 a strongly measurable function $u:[0,T]\to X$ with the following properties:
 \begin{enumerate}
 \item[(i)] $u$ takes values in $\Dom(A)$ almost everywhere and $Au$ belongs to $L^p(0,T;X)$;
 \item[(ii)] for almost all $t\in (0,T)$ we have
 $$ u(t) + \int_0^t Au(s)\ud s = \int_0^t f(s)\ud s;$$
 \item[(iii)] we have the estimate $$\n Au\n_{L^p(I;X)} \le C \n f\n_{L^p(0,T;X)},$$
 with a constant $C\ge 0$ independent of $f$.
 \end{enumerate}
A systematic discussion of maximal $L^p$-regularity is given in \cite{Dore}, where among other things it is shown that if $A$ has maximal $L^p$-regularity, then
$A$ generates an (analytic) $C_0$-semigroup. In particular, maximal $L^p$-regularity implies that (H4) holds.
A celebrated result of Weis \cite{Weis} states that a densely defined closed operator $A$ in a UMD space $X$ has maximal $L^p$-regularity and only if 
$-A$ generates an analytic $C_0$-semigroup on $X$ which is $\gamma$-bounded on some sector in the complex plane containing the positive real axis. 
In particular this implies that (H5) holds.

Examples of operators with maximal $L^p$-regularity include most second-order elliptic operators on $\R^d$ or on sufficiently smooth bounded domains in $\R^d$ with 
various boundary conditions, provided the coefficients satisfy appropriate smoothness assumptions. 
For more details the reader is referred to \cite{DHP, Dore, HNVW3, KuWe, PruSim}.

Below we will consider the three special cases where 
(a) $A=0$ and the process $f:[0,T]\times\Omega\times X\times X\to X$ only depends on the first two variables, (b) the process $f:[0,T]\times\Omega\times X\times X\to X$ only depends on the first two variables, and (c) no additional restrictions are imposed.
 The precise assumptions on $f$ will depend on the case under consideration, but in each of the three cases they coincide with, 
or are special cases of, the following condition: 

\begin{enumerate}
 \item [\rm(H6)] The function $f: [0,T]\times\Omega\times X\times X\to X$ has the following properties:
 \begin{enumerate}
 \item[(i)] $f$ is jointly measurable in the first two variables and continuous in the third and fourth;
 
 \item[(ii)] for all 
 $U,V\in L_\F^p(\Omega;\gamma(0,T;X))$  the  process  $$f(\cdot,U,V): (t,\omega) \mapsto f(t,\omega,U(t,\omega),V(t,\omega))$$ defines an element of $L_{\F}^p(\Om;\gamma(0,T;X))$;
 
 \item[(iii)]
 and there is a constant $C\ge 0$  such that  for all 
 $U,V\in L_\F^p(\Omega;\gamma(0,T;X))$ we have
\begin{align*} \ \qquad\quad \n & f(\cdot,U,V)\n_{L^p(\Omega;\gamma(0,T;X))} 
\\ & \qquad\qquad \le C(1+\n U\n_{L^p(\Omega;\gamma(0,T;X))} + \n V\n_{L^p(\Omega;\gamma(0,T;X))});
\end{align*}

 \item[(iii)] there is a constant $L\ge 0$  such that for all 
 $U,U',V,V'\!\in L_\F^p(\Omega;\gamma(0,T;X))$ we have
 \begin{align*}
  \ \qquad\qquad\quad\ \ \n & f(\cdot,U,V) - f(\cdot,U',V')\n_{L^p(\Omega;\gamma(0,T;X))} 
\\ & \qquad\qquad \le L(\n U-U'\n_{L^p(\Omega;\gamma(0,T;X))} + \n V-V'\n_{L^p(\Omega;\gamma(0,T;X))}).
 \end{align*}
\end{enumerate} 
\end{enumerate}

A closely related notion of $\ga$-Lipschitz
continuity has been introduced and studied in
\cite{NVW08}. In the same way as in this
reference once shows that if $X$ has type $2$
(e.g., if $X$ is a Hilbert space or a space
$L^p(\mu)$ with $2\le p<\infty$), then the usual
linear growth and Lipschitz conditions
\begin{equation*}
\begin{aligned}
\n f(t,\om,x,y)\n & \le C_f(1+\n x\n + \n y\n), \\
\n f(t,\om,x,y) - f(t,x',y')\n & \le L_f (\n x-x'\n
+ \n y-y'\n),
\end{aligned}
\end{equation*}
imply that $f$ satisfies (H6).

\begin{definition}\label{def:Lp-solution} Assume (H1)--(H6).
A {\em mild $L^p$-solution} to the problem \eqref{eq:BSDE-intr} is a pair $(U,V)$, where $U$ and $V$ are continuous $\F$-adapted processes defining elements in $L_{\F}^p(\Om;\ggX)$
such that
$$ U(t) + \int_t^T S(s-t)f(s,U(s),V(s))\ud s + \int_t^T S(s-t)V(s)\ud W(s) = S(T-t)u_T,$$
where the identity is to be interpreted in the sense explained in Subsection \ref{subsec:radonif}.
\end{definition}

Assumptions (H5) and (H6) imply, via the Kalton--Weis multiplier theorem, that if $U,V\in L_{\F}^p(\Om;\ggX)$, then for each $t\in [0,T]$ the mappings
$s\mapsto  S(s-t)f(s,U(s),V(s))$ and $s\mapsto S(s-t)V(s)$ define elements in $L_{\F}^p(\Om;\gamma(t,T;X))$.
Therefore by \eqref{eq:est-t} the integral $$\int_t^T S(s-t)f(s,U(s),V(s))\ud s$$ 
is well defined as an element of $L^p(\Om;X)$, and by Theorem \ref{thm:Itoisomorph} the same is true for the stochastic integral
$$  \int_t^T S(s-t)V(s)\ud W(s) .$$
Thus, in hindsight, the identity in Definition \ref{def:Lp-solution} admits an interpretation in $L^p(\Om;X)$ pointwise in $t\in [0,T]$, and it is of interest to ask about time regularity of $U$.

\begin{proposition} Assume (H1)--(H6).
 If $(U,V)$ is a mild $L^p$-solution to the problem \eqref{eq:BSDE-intr}, then $U$ belongs to $C([0,T];L^p(\Om;X))$.
\end{proposition}
\begin{proof}
It is not hard to see 
that $t\mapsto \int_t^T S(s-t)f(s,U(s),V(s))\ud s$ belongs to $L^p(\Om;C([0,T];X))$ (and hence to $C([0,T];L^p(\Om;X))$). Indeed, arguing pathwise, it suffices to note that 
for all $g$ in the dense subspace $L^2(0,T)\otimes X$ of 
$\gamma(0,T;X)$ the mapping $t\mapsto \int_t^T S(s-t)g(s)\ud s$ is continuous 
and satisfies
\begin{align*}
\ & \sup_{t\in [0,T]} \Big\n \int_t^T S(s-t)g(s)\ud s\Big\n \\ & \qquad \le\sup_{t\in [0,T]} (T-t)^{1/2} \n s\mapsto S(t-s)g(s)\n_{\gamma(T-t,T;X)}
\le T^{1/2}\gamma(S)\n g\n_{\gamma(0,T;X)}
\end{align*} using \eqref{eq:est-t},
where 
$\gamma(S)$ is the $\gamma$-bound of $\{S(t):\, t\in [0,T]\}$.
Similarly the mapping  $t\mapsto \int_t^T S(s-t)V(s)\ud W(s)$ is seen to belong to $C([0,T];L^p(\Om;X))$.
Indeed for adapted $X$-valued step processes $V$, which are dense in
$L_\F^p(\Om,\ggX)$, the mapping  $t\mapsto \int_t^T S(s-t)V(s)\ud W(s)$ is continuous and satisfies
\begin{align*}
\ &  \sup_{t\in [0,T]} \Big\n t\mapsto \int_t^T S(s-t)V(s)\ud W(s) \Big\n_{L^p(\Om;X)} 
\\ & \qquad \lesssim_{p,X} \sup_{t\in [0,T]} \n s\mapsto  S(s-t)V(s)\n_{L^p(\Om;\gamma(T-t,T;X))}
 \le \gamma(S)\n V\n_{L^p(\Om;\gamma(0,T;X))}
\end{align*}
using Theorem \ref{thm:Itoisomorph}. 
\end{proof}

From the proof we see that $U$ is in $L^p(\Om;C([0,T];X))$ if and only if $t\mapsto \int_t^T S(s-t)V(s)\ud W(s)$ 
is in $L^p(\Om;C([0,T];X))$, but the latter is not to be expected unless we make 
additional conditions implying maximal estimates for stochastic convolutions (such as in \cite[Section 4]{VerWei}).

\subsection{The case $A=0$, $f(t,\om,x,y) = f(t,\om)$}\label{subsec:A=0}
We begin by considering the problem
\begin{equation}\label{eq:BSDE0}
\left\{
\begin{aligned}
 \ud U(t) & = f(t) \ud t + V(t)\ud W(t), \quad t\in [0,T], \\
  U(T) & = u_T,
  \end{aligned}\right.
\end{equation}
assuming (H1)--(H3) as well as 

\medskip\noindent
(H6)$'$ \  $f$ defines an element of $L_{\F}^p(\Om;\gamma(0,T;X))$.

\medskip\noindent
We comment on this assumption in Remark \ref{rem:H6} below.
Even though \eqref{eq:BSDE0} is a special case of the problem \eqref{eq:BSDE1} considered in the next subsection, it is instructive to treat it separately. 

Following the ideas of \cite{PP} we define
the $X$-valued process $M$  by
$$ M(t): = \E\Bigl( u_T - \int_0^T f(s)\ud s \Big| \calF_t\Bigr).$$
By \cite[Theorems 4.7, 5.13]{NVW07} this is a continuous $L^p$-martingale with respect to $\F$ in $X$ and there exists a unique $V\in L_{\F}^p(\Om;\ggX)$ such that
\begin{equation}\label{eq:V} M(t) = M(0) + \int_0^t V\ud W.
\end{equation}
By \cite[Theorems 4.5, 5.12]{NVW07} and the observations in Subsection \ref{subsec:radonif}
combined with Lemma \ref{lem:adapted}, both $M$ and the $\F$-adapted process
\begin{equation}\label{eq:U}  U(t) := M(t)+ \int_0^t f(s)\ud s
\end{equation}
belong to $L_\F^p(\Omega;\gamma(0,T;X))$. 

\begin{proposition}\label{prop:simplecase1} Let (H1)-(H3) and (H6)$'$ be satisfied. Then the problem
\eqref{eq:BSDE0} admits a unique mild $L^p$-solution $(U,V)$. It is given by 
the pair constructed in \eqref{eq:V} and \eqref{eq:U}.
\end{proposition}
\begin{proof}
Let $U$ and $V$ be defined by  \eqref{eq:V} and \eqref{eq:U}.
We have already checked that $U$
and $V$ belong to $L_{\F}^p(\Om;\ggX)$.
To show that $(U,V)$ is an $L^p$-solution, note that
\begin{equation*}
\begin{aligned}
\ &  U(t) + \int_t^T f(s)\ud s + \int_t^T V\ud W
\\ &\qquad  = \Bigl( M(t)+ \int_0^t f(s)\ud s \Bigr) + \int_t^T f(s)\ud s + (M(T) - M(t))
\\ &\qquad  = \int_0^T f(s)\ud s + M(T)
\\ &\qquad  = \int_0^T f(s)\ud s +\Bigl(u_T - \int_0^T f(s)\ud s \Big)
\\ & \qquad = u_T.
 \end{aligned}
 \end{equation*}

Concerning uniqueness, suppose $(\wt U, \wt V)$ is another $L^p$-solution.
Then
\begin{equation}\label{eq:unique} \wt U(t) - U(t) + \int_t^T (\wt V - V)\ud W = 0 \quad \forall t\in [0,T].
\end{equation}
Taking conditional expectations with respect to $\calF_t$ it follows that
$ \wt U(t) - U(t) = 0$, where we used \cite[Proposition 4.3]{NVW07} to see that the conditional
expectation of the stochastic integral vanishes. Uniqueness of $V$ is already implicit in the uniqueness part of \eqref{eq:V}. 
It also follows from \eqref{eq:unique}, where $\wt U = U$ gives
$\int_t^T (\wt V - V)\ud W = 0 $ for all $t\in [0,T]$. Taking $t=0$
and taking $L^p$-means, using \cite[Theorem 3.5]{NVW07} it follows that
$$\n \wt V - V\n_{L^p(\Om;\ggX)}\eqsim_{p,X} \E \Big\n\int_0^T (\wt V - V)\ud W\Big\n^p  = 0,$$
and therefore $\wt V = V$ in $L^p(\Om;\ggX)$.
\end{proof}

\begin{remark}\label{rem:H6}
 The reader may check that, {\em mutatis mutandis}, Proposition \ref{prop:simplecase1} admits a version when (H6)$'$ 
is replaced by the simpler condition $f\in L_{\F}^p(\Om;L^1(0,T;X))$. That the integral in \eqref{eq:U} 
defines an element of  $L_\F^p(\Omega;\gamma(0,T;X))$ then follows from \cite[Proposition 9.7.1]{HNVW2} .
The motivation for the present formulation of (H6)$'$ is that it is a special case
 of the assumption (H6) needed in the final section where mixed $L^p$-$L^1$ conditions do not seem to work.
\end{remark}

\subsection{The case $f(t,\om,x,y)=f(t,\om)$}

We now consider the problem
\begin{equation}\label{eq:BSDE1}
\left\{\begin{aligned}
 \ud U(t) +AU(t)\ud t & = f(t)\ud t + V(t)\ud W(t), \quad t\in [0,T], \\
  U(T) & = u_T,
\end{aligned}\right.
\end{equation}
assuming (H1)--(H4) and (H6)$'$. 
Our proof of the well-posedness of the problem \eqref{eq:BSDE1} relies on the following lemma,
where $s$ and $\sigma$ denote two time variables; the dependence on $\omega$ is suppressed.
To give a meaning to the expression in the second condition below we recall from \eqref{eq-gammaLpX} the isomorphism of Banach spaces
$$ \gamma(0,T;L^p(\Om;Y)) \eqsim_p L^p(\Om;\gamma(0,T;Y)).$$ This isomorphism allows us to interpret, in condition (2) below, $k$ as an element of 
$\gamma(0,T;L_{\F}^p(\Om;\ggX))$.

\begin{lemma}\label{ch-1-lemma1} Let (H1), (H2), and (H6)$'$ be satisfied. There exists a unique 
$k\in L_{\F}^p(\Om;\gamma(0,T;\ggX))$ satisfying
the following conditions:
\begin{enumerate}
\item[\rm(1)] almost surely, $k$ is supported on the set
$\{(s,\sigma)\in [0,T]\times[0,T]: \ \sigma\le s\}$;

\item[\rm(2)]\label{ch-1-lemma1-eq1} for almost all $s\in [0,T]$ we have
\begin{equation*}
f(s) = \E f(s) + \int_0^s k(s,\sigma)\ud W(\sigma)
\ \ \hbox{in $L^p(\Om;X)$};
\end{equation*}

\item[\rm(3)]
we have the estimate 
\begin{equation*}
\n k\n_{L^p(\Om;\gamma(0,T;\ggX))}\lesssim_{p,X}
\n f\n_{L^p(\Om;\ggX))}.
\end{equation*}
\end{enumerate}
\end{lemma}
The precise meaning of condition (1) is that for almost all $\om\in \Om$,
the operator $k(\om)\in \gamma(0,T;\ggX)$ vanishes on all $f\in L^2(0,T)\otimes L^2(0,T)$, which, as 
 functions on $(0,T)\times (0,T)$, are supported
on the set  $\{(s,\sigma)\in (0,T)\times[0,T]: \ \sigma > s\}$.

\begin{proof} Since by assumption $f\in L_\F^p(\Om;\ggX)$, 
we may pick a sequence of adapted step
processes $\{f_n\}_{n=1}^\infty$ such that $f_n\to f$ in
$L^p(\Om;\ggX))$ as $n\to\infty$.
For each $n\ge 1$ we then may write $$ f_n(s,\om)=\sum_{i=0}^{N_n-1}
\one_{[t_{n,i},t_{n,i+1})}(s)\xi_{n,i}(\om)$$ where
$\{t_{n,0},t_{n,1},\cdots,t_{n,N_n}\}$ is a partition of $[0,T]$
and the random variables $\xi_{n,i}\in
L^p(\Omega;X)$ are strongly $\calF_{t_{n,i}}$-measurable.
By \cite[Theorem 3.5]{NVW07} there exist $k_{n,i}\in
L^p_{\F}(\Om;\ga(0,t_{n,i};X))$ such that
\begin{equation*}
\xi_{n,i} = \E\xi_{n,i} + \int_0^{t_{n,i}}k_{n,i}\ud W.
\end{equation*}
In what follows we will identify $k_{n,i}$ with elements of
$L^p_{\F}(\Om;\ga(0,T;X))$ in the natural way.
Put
$$ k_n(s,\sigma) := \sum_{i=0}^{N_n-1}
\one_{[t_{n,i},t_{n,i+1})}(s)\one_{[0,t_{n,i})}(\sigma)
k_{n,i}(\sigma).
$$
Each $k_n$ satisfies the support condition of (1) and
\begin{equation}\label{2.15-eq3}
f_n(s) = \E f_n(s) + \int_0^s k_n(s,\sigma)\ud W(\sigma).
\end{equation}
Choose an orthonormal basis $\{h_j\}_{j\ge 1}$ for $L^2(0,T)$
and let $\{\gamma_j'\}_{j\ge 1}$ be a Gaussian sequence on an independent probability space
$(\Omega',\P')$. Then,
by \cite[Theorem 9.1.17]{HNVW2}, the It\^o isomorphism of Theorem \ref{thm:Itoisomorph}, 
and the stochastic Fubini theorem (see, e.g., \cite{NeeVer}) and keeping in mind the support properties, we have
\begin{equation}\label{2.15-eq4}
\begin{aligned}
\ & \big\n s\mapsto k_n(s,\cdot)-k_m(s,\cdot)\big\n_{\gamma(0,T;L^p(\Om;\ggX))}^p 
\\ & \qquad \eqsim_{p} \E' \Big\n\sum_{j\ge 1} \gamma_j'  \int_0^T h_j(s)(k_n(s,\cdot)-k_m(s,\cdot))\ud s \Big\n_{L^p(\Om;\ggX)}^p 
\\ & \qquad \eqsim_{p,X} 
\E'\E \Big\n\sum_{j\ge 1} \gamma_j' \int_0^T \int_0^T h_j(s)(k_n(s,\sigma)-k_m(s,\sigma))\ud s\ud W(\sigma) \Big\n^p
\\ & \qquad = 
\E\E' \Big\n\sum_{j\ge 1} \gamma_j' \int_0^T h_j(s)\int_0^s (k_n(s,\sigma)-k_m(s,\sigma))\ud W(\sigma)\ud s \Big\n^p
\\ & \qquad \eqsim_{p} 
\E \Big\n s\mapsto \int_0^s (k_n(s,\sigma)-k_m(s,\sigma))\ud W(\sigma) \Big\n_{\gamma(0,T;X)}^p
\\ & \qquad =
\E \big\n s\mapsto [f_n(s)-f_m(s) - (\E f_n(s)-\E f_m(s))] \big\n_{\gamma(0,T;X)}^p,
\end{aligned}
\end{equation}
and therefore
$$ \big\n s\mapsto k_n(s,\cdot)-k_m(s,\cdot)\big\n_{\gamma(0,T;L^p(\Om;\ggX))}
\lesssim_{p,X}\n f_n -f_m\n_{L^p(\Omega;\gamma(0,T;X))}.
$$
Since $\{f_n\}_{n=1}^\infty$ is a Cauchy
sequence in $\gamma(0,T;L^p(\Omega;X))$, the estimate \eqref{2.15-eq4} implies that
$\{k_n\}_{n=1}^\infty$ is a Cauchy sequence in
$\gamma(0,T;L^p(\Om;\ggX))$. Let $k\in \gamma(0,T;L^p(\Om;\ggX)) 
\eqsim L^p(\Om; \gamma(0,T;\ggX))$ be its limit.
By adaptedness of the $k_n$ we have $L_\F^p(\Om; \gamma(0,T;\ggX))$, and 
by passing to the limit $n\to\infty$ in \eqref{2.15-eq3}, assertions (1) and (2) are obtained.

Similar to \eqref{2.15-eq4} we have
\begin{equation}\label{2.15-116eq4}
\n s\mapsto k_n(s,\cdot)\n_{\gamma(0,T;L^p(\Om;\ggX))} 
\lesssim_{p,X}\n f_n \n_{\gamma(0,T;L^p(\Omega;X))}.
\end{equation}
Letting $n\to\infty$ in \eqref{2.15-116eq4} we
obtain assertion (3).
\end{proof}

\begin{proposition}\label{prop:simplecase2}
Let (H1)--(H5) and (H6)$'$ be satisfied and assume in addition that $X$ has the upper contraction property. Then the problem
\eqref{eq:BSDE1} admits a unique mild 
$L^p$-solution $(U,V)$.
\end{proposition}

\begin{proof} We extend the argument of \cite{HP} to the UMD setting. As in Subsection \ref{subsec:A=0},
by martingale representation in UMD spaces there
is a unique element $\phi\in L_{\F}^p(\Om;\ggX)$ such
that for all $t\in [0,T]$,
\begin{equation}\label{12.17-eq1}
\E( u_T|\calF_t)=\E u_T + \int_0^t\phi \ud W \ \ \hbox{in $L^p(\Om;X)$}.
\end{equation}
Put
\begin{equation*}
U(t):=\E\Big(S(T-t) u_T - \int_t^T S(s-t)f(s)\ud s\;\Big|\;\calF_t \Big).
\end{equation*}
Let
$k\in L_{\F}^p(\Om;\gamma(0,T;\ggX))$ be the kernel obtained from
Lemma \ref{ch-1-lemma1}. Then for almost all $s\in [0,T]$ we have
\begin{equation}\label{12.17-eq2}
f(s) =\E f(s) + \int_0^s k(s,\sigma)\ud W(\sigma).
\end{equation}

By \eqref{12.17-eq1} (applied to $t$ and $T$ and subtracting the results),
\begin{align}\label{12.17-eq1a}
u_T - \E(u_T|\calF_t) = \int_t^T \phi\ud W. 
\end{align}
The definition of $U$, together with \eqref{12.17-eq2} and \eqref{12.17-eq1a},
 implies that
\begin{equation}\label{12.17-eq3}
\begin{aligned}
U(t) & =  \E(S(T-t) u_T|\calF_t)
-\Bigl(\int_t^T S(s-t) \Bigl(\E f(s)+\int_0^s k(s,\sigma)\ud W(\sigma)\Bigr)\Big|\calF_t\Bigr)\ud s
\\ & =  S(T-t)\E(u_T|\calF_t)
- \int_t^T S(s-t)\Bigl(\E f(s)+\int_0^t k(s,\sigma)\ud W(\sigma)\Bigr)\ud s
\\ & =  S(T-t)\Big(u_T - \int_t^T \phi\ud W\Big)
- \int_t^T S(s-t) \Bigl(f(s)-\int_t^s k(s,\sigma)\ud W(\sigma)\Bigr)\ud s.
 \end{aligned}
 \end{equation}
We will analyse the two terms on the right-hand side separately.

Since by assumption $\{S(t):\ t\in [0,T]\}$ is  $\gamma$-bounded, we may apply the Kalton-Weis multiplier theorem
(\cite[Theorem 9.5.1]{HNVW2})
to see that
$t\mapsto S(T-t)\E u_T$ defines an element of  $L^p(\Om,\ggX)$. By Lemma \ref{lem:adapted}
it then defines an element of  $L_{\F}^p(\Om,\ggX).$
Also,
by \cite[Theorem 4.5]{NVW07}, $t\mapsto \int_t^T \phi\ud W$ defines an element of  $L^p(\Om,\ggX)$,
and by another appeal to $\gamma$-boundedness, the same is true for  $$t\mapsto S(T-t) \int_t^T \phi\ud W.$$
By Lemma \ref{lem:adapted} this mapping defines an element of  $L_{\F}^p(\Om,\ggX).$

We now turn to the second term in the right-hand side of \eqref{12.17-eq3} and consider the two terms in the integral separately. For the first term we observe that 
 $$t\mapsto \int_t^T S(s-t)f(s)\ud s$$ belongs
to $L^p(\Om;\gamma(0,T;X))$ by Lemma \ref{lem:gammaDeterm}(1).
Turning to the second term in the integral, to see that the mapping 
\begin{align*}
t\mapsto \int_t^T S(s-t)\int_t^s  k(s,\sigma)\ud W(\sigma)\ud s
 \end{align*}
defines an element of  $L^p(\Om;\gamma(0,T;X))$ we apply 
 the stochastic Fubini theorem, the isomorphism $L^p(\Om;\gamma(0,T;X))
\eqsim \gamma (0,T;L^p(\Om;X))$, 
Theorem \ref{thm:Itoisomorph}, the isomorphism once more,
Lemma \ref{lem:gammaDeterm}(2), the Kalton--Weis multiplier theorem, and the upper contraction property. This leads to the estimate
\begin{equation}\label{eq:IV}
\begin{aligned}
 & \Big\n t\mapsto \int_t^T S(s-t)\int_t^s k(s,\sigma)\ud W(\sigma)\ud s\Big\n_{L^p(\Om;\gamma(0,T;X))}
\\ & \qquad =
 \Big\n t\mapsto\int_t^T  \int_\sigma^T S(s-t)k(s,\sigma)\ud s\ud W(\sigma)\Big\n_{L^p(\Om;\gamma(0,T;X))}
\\ & \qquad \eqsim_{p,X} \Big\n t\mapsto\int_t^T  \int_\sigma^T S(s-t)k(s,\sigma)\ud s\ud W(\sigma)\Big\n_{\gamma (0,T;L^p(\Om;X))}
\\ & \qquad \eqsim_{p,X}\Big\n t\mapsto\Bigl[\sigma\mapsto \int_\sigma^T S(s-t)k(s,\sigma)\ud s\Big]\Big\n_{\gamma (0,T;L^p(\Om;\gamma(0,T;X)))}
\\ & \qquad \eqsim_{p,X}\Big\n t\mapsto\Bigl[\sigma\mapsto \one_{\{t\le \sigma\}} S(\sigma-t)\int_\sigma^T S(s-\sigma)k(s,\sigma)\ud s\Big]\Big\n_{L^p(\Om;\gamma(0,T;\gamma(0,T;X)))}
\\ & \qquad \le \gamma(S)\Big\n t\mapsto\Bigl[\sigma\mapsto \int_\sigma^T S(s-\sigma)k(s,\sigma)\ud s\Big]\Big\n_{L^p(\Om;\gamma (0,T;\gamma(0,T;X)))}
\\ & \qquad = T^{1/2} \gamma(S)\Big\n\sigma\mapsto \int_\sigma^T S(s-\sigma)k(s,\sigma)\ud s\Big\n_{L^p(\Om;\gamma(0,T;X)))}
\\ & \qquad \lesssim_{p,X}  T \gamma(S)^2 \n k\n_{L^p(\Om;\gamma(\Delta;X))},
\\ & \qquad \eqsim_{p,X}  T \gamma(S)^2 \n k\n_{L^p(\Om;\gamma(0,T;\gamma(0,T;X)))}.
\end{aligned}
\end{equation}

Collecting what has been proved, it follows that  $U\in L_{\F}^p(\Om;\ggX)$, the
adaptedness of $U$ being a consequence of Lemma \ref{lem:adapted} and the representation given by the first identity in \eqref{12.17-eq3}.

By the stochastic Fubini theorem, 
 \begin{align*}
 U(t) & =S(T-t) u_T - \int_t^T S(s-t)f(s)\ud s - \int_t^T S(T-t)\phi(\sigma)\ud W(\sigma)
\\ & \qquad  + \int_t^T\int_\sigma^T S(s-t)k(s,\sigma)\ud s \ud W(\sigma)
\\ & = S(T-t) u_T - \int_t^T S(s-t) f(s) \ud s - \int_t^T S(\sigma-t)V(\sigma)\ud W(\sigma),
\end{align*}
where
\begin{equation}\label{12.17-eq5}
\sigma\mapsto V(\sigma):=S(T-\sigma)\phi(\sigma)+\int_\sigma^T S(s-\sigma)k(s,\sigma)\ud s
\end{equation}
is $\F$-adapted.
It remains to be checked that the process $V$ defines an element of $L_{\F}^p(\Om;\ga(0,T;X))$. This can be done by repeating the arguments which showed the corresponding result for $U$.

\smallskip
Next we prove the uniqueness of the solution. The
proof is very similar to the one for  $A=0$.
Suppose $(\wt U, \wt V)$ is another
$L^p$-solution to \eqref{eq:BSDE1}. Then from
the definition of the mild solution to
\eqref{eq:BSDE1}, we find that
\begin{equation}\label{eq:unique1}
\wt U(t) - U(t) + \int_t^T S(s-t)(\wt V(s) - V(s))\ud
W(s) = 0 
\end{equation}
for all $t\in [0,T]$
By taking conditional expectations with respect
to $\calF_t$ for \eqref{eq:unique1}, we see that
$ \wt U(t) - U(t) = 0$. Thus $\int_t^T
S(s-t)(\wt V(s) - V(s))\ud W(s) = 0 $ for all $t\in
[0,T]$. Taking  $L^p$-means, using \cite[Theorem
3.5]{NVW07} it follows that
$$\n S(\cdot-t)(\wt V(\cdot) - V(\cdot))\n_{L^p(\Om;\ggX)}^p\eqsim_{p,X} \E \Big\n\int_t^T S(s-t)(\wt V(s) - V(s))\ud W(s)\Big\n^p  = 0.$$
Hence, for any $t\in [0,T]$, in $L^p(\Om;\gamma(t,T;X))$ we obtain the equality
$$
S(\cdot-t) \wt V(\cdot) = S(\cdot-t) V(\cdot).
$$
To deduce from this that $\wt V=V$ in $L^p(\Om;\ggX)$ we argue pathwise and prove that if $v\in \gamma(0,T)$
satisfies $S(\cdot-t) v(\cdot) = 0$ in $\gamma(t,T)$ for all $t\in [0,T]$, then $v=0$. Fix an integer $N\ge 1$ and set $t_j = jT/N$ for $j=0,1,\dots,N$. Multiplying the identity $S(\cdot-t_j) v(\cdot) = 0$ by
$S(t_{j+1}-(\cdot-t_j))$ on $I_j:= [t_j,t_{j+1}]$ it follows that $S(T/N)v(\cdot) = 0$ as an element of
$\gamma(t_j,t_{j+1};X)$, $j=0,1,\dots,N-1$, 
and therefore $S(T/N)v(\cdot) = 0$ as an element of $\gamma(0,T;X)$. Now we can apply \cite[Proposition 9.4.6]{HNVW2} to deduce that $v = 0$ as an element of $\gamma(0,T;X)$.
\end{proof}

\subsection{The general case}\label{subsec:generalcase}

In the final section we consider the problem
\begin{equation}\label{eq:BSDE2}
\left\{\begin{aligned}
 {\rm d}U(t) + AU(t)\ud t & = f(t,U(t),V(t))\ud t +  V(t)\ud W(t), \quad t\in [0,T], \\
  U(T) & = u_T,
  \end{aligned}\right.
\end{equation}
under the assumptions (H1)--(H6).

\begin{theorem} Let (H1)--(H6) be satisfied and assume in addition that $X$ has the upper contraction property. 
Then the problem \eqref{eq:BSDE2} admits a unique mild $L^p$-solution $(U,V)$.
\end{theorem}
\begin{proof}
Following the ideas of \cite{PP} the existence proof proceeds by a Picard iteration argument,
where the existence and uniqueness in each iteration follows from the well-posedness of the 
problem \eqref{eq:BSDE1} considered in the previous subsection. 

\smallskip
{\em Step 1} -- 
In this step we prove the existence of an $L^p$-solution
on the interval $I_\delta:= [T-\delta,T]$ for $\delta\in (0,T)$ small enough. 

Set $U_0=0$ and $V_0=0$ and define the pair $(U_{n+1},V_{n+1}) \in L^p_{\F}(\Om;\g(I_\delta;X))\times L^p_{\F}(\Om;\g(I_\delta;X))$  inductively
as the unique mild $L^p$-solution of the problem
\begin{equation*}\left\{
\begin{aligned}
 {\rm d}U(t) & = - AU(t)\ud t+f(t,U_{n}(t), V_{n}(t))\ud t + V_{n}(t)\ud W(t), \quad t\in I_\delta, \\
  U(T) & = u_T.
\end{aligned}
\right.
\end{equation*}
Note that at each iteration the function $t\mapsto g_n(t) := f(t,U_{n}(t), V_{n}(t))$ defines an element of  $ L^p_{\F}(\Om;\g(I_\delta;X))$ by (H6)
with norm 
$$ \n g_n\n_{L^p_{\F}(\Om;\g(I_\delta;X))} \le C(1+\n U_{n}\n_{L^p_{\F}(\Om;\g(I_\delta;X))} + \n V_n\n_{L^p_{\F}(\Om;\g(I_\delta;X))})$$
with a constant $C\ge 0$ independent of $U_n$ and $V_n$.
By Proposition \ref{prop:simplecase2},
\begin{align*}
 \n U_1 - U_0\n_{L^p_{\F}(\Om;\g(I_\delta;X))} = \n U_1\n_{L^p_{\F}(\Om;\g(I_\delta;X))}
 &\le C (\n g_0\n_{L^p_{\F}(\Om;\g(I_\delta;X))} + \n u_T\n_{L^p(\Om;X)}), \\
  \n V_1-V_0\n_{L^p_{\F}(\Om;\g(I_\delta;X))}=  \n V_1\n_{L^p_{\F}(\Om;\g(I_\delta;X))} & \le C(\n g_0\n_{L^p_{\F}(\Om;\g(I_\delta;X))}+\n u_T\n_{L^p(\Om;X)}),
\end{align*}
where $C\ge 0$ is a constant independent of $f$ and $u_T$.

For $n\ge 1$, by \eqref{12.17-eq3} we can estimate 
\begin{equation*}
\begin{aligned}
\ & \n U_{n+1} - U_{n}\n_{L^p_{\F}(\Om;\g(I_\delta;X))} 
\\ & \ \, \le
\Big\n t\mapsto 
\int_t^T S(s-t) (g_n(s)-g_{n-1}(s)) \ud s\Big\n_{L^p_{\F}(\Om;\g(I_\delta;X))}
\\ & \qquad + \Big\n t\mapsto\int_t^T S(s-t) \int_t^s (k_n(s,\sigma)-k_{n-1}(s,\sigma))\ud W(\sigma)\ud s\Big\n_{L^p_{\F}(\Om;\g(I_\delta;X))}
\\ &  = (I)+(II).
\end{aligned}
\end{equation*}
We estimate these terms separately. To to estimate (I) we use Lemma \ref{lem:gammaDeterm}(1) with $[0,T]$ replaced by $I_\delta$:
\begin{align*}
 (I) & = \Big\n t\mapsto \int_t^T S(s-t) (g_n(s)-g_{n-1}(s)) \ud s\Big\n_{L^p(\Om;\gamma(I_\delta;X))}
\\ & \le  \delta \gamma(S)\n g_n-g_{n-1}\n_{L^p(\Om;\gamma(I_\delta;X))}
\\ & \le L\delta \gamma(S)
(\n U_n- U_{n-1}\n_{L^p_{\F}(\Om;\g(I_\delta;X))} + \n V_n - V_{n-1}\n_{L^p_{\F}(\Om;\g(I_\delta;X))}),
\intertext{where $\gamma(S)$ is the $\gamma$-bound of $\{S(t):\, t\in [0,T]\}$ 
and $L$ the Lipschitz constant in (H6).
To estimate (II) we proceed as in \eqref{eq:IV}, again with $[0,T]$ replaced by $I_\delta$:}
 (II) & = 
\Big\n t\mapsto\int_t^T S(s-t) \int_t^s (k_n(s,\sigma)-k_{n-1}(s,\sigma))\ud W(\sigma)\ud s\Big\n_{L^p(\Om;\gamma(I_\delta;X))}
\\ & \le \delta^{1/2} \gamma(S)\Big\n\sigma\mapsto \int_\sigma^T S(s-\sigma)(k_n(s,\sigma)-k_{n-1}(s,\sigma))\ud s\Big\n_{L^p(\Om;\gamma(I_\delta;X)))}
\\ & =  \delta^{1/2} \gamma(S) \n V_{n+1}-V_n\n_{L^p_{\F}(\Om;\g(I_\delta;X))},
\end{align*}
using \eqref{12.17-eq2} and \eqref{12.17-eq5} in the last step. Moreover, by Lemmas \ref{lem:gammaDeterm}(2)
and \ref{lem:triangular}, and \ref{ch-1-lemma1},
\begin{align*} 
\ & \n V_{n+1}-V_n\n_{L^p_{\F}(\Om;\g(I_\delta;X))}
 \\ & \qquad \le \delta^{1/2}\gamma(S)\n k_n-k_{n-1}\n_{L^p(\Om;\gamma(\Delta_\delta;X))}
 \\ & \qquad \eqsim_{p,X} \delta^{1/2}\gamma(S)\n k_n-k_{n-1}\n_{L^p(\Om;\gamma(I_\delta;\gamma(I_\delta;X)))}
 \\ & \qquad \lesssim_{p,X}\delta^{1/2}\gamma(S)\n g_n - g_{n-1}\n_{L^p_{\F}(\Om;\g(I_\delta;X))} 
 \\ & \qquad = \delta^{1/2}\gamma(S) \n  f(\cdot,U_{n}(\cdot), V_{n}(\cdot)) -  f(\cdot,U_{n-1}(\cdot), V_{n-1}(\cdot))\n_{L^p_{\F}(\Om;\g(I_\delta;X))} 
 \\ & \qquad \le L\delta^{1/2}\gamma(S)(\n U_n-U_{n-1}\n_{L^p_{\F}(\Om;\g(I_\delta;X))} + \n V_n-V_{n-1}\n_{L^p_{\F}(\Om;\g(I_\delta;X))}).
\end{align*}

Combining all estimates, we see that, if $\delta$ is small enough, the sequences
$\{U_n\}_{n\ge 1}$ and $\{V_n\}_{n\ge 1}$ converge in 
$L^p_{\F}(\Om;\g(I_\delta;X))$ to limits $U$ and $V$. It is clear that the pair $(U,V)$ is an $L^p$-solution on the interval $I_\delta$. 

\smallskip
{\em Step 2} -- The arguments in Step 1 show
that we  always obtain a unique mild
$L^p$-solution if $\d$ is small enough. 
Since the estimates involve constants that are
independent of $T$, $\delta$, and $u_T$, 
the proof may be repeated with
$I_\delta$ replaced by any interval
$[T-2\d,T-\d]$. In this way we can obtain a
global existence result by partitioning $[0,T]$
into finitely many such intervals, and the
successively solving the backwards equation
proceeding `from the right to the left'. This
gives us solutions for the backward equation on
each sub-interval, and it is easy to check that
a global solution is obtained by patching
together these local solutions.

\smallskip
{\em Step 3} -- 
Finally we prove the uniqueness of the solution. The
proof is very similar to the one for  $A=0$.
Suppose $(\wt U, \wt V)$ is another
$L^p$-solution to \eqref{eq:BSDE2}. Then from
the definition of the mild solution to
\eqref{eq:BSDE2}, we find that
\begin{equation}\label{eq:unique2}
\wt U(t) - U(t) + \int_t^T S(s-t)(\wt V(s) - V(s))\ud
W(s) = 0 
\end{equation}
for all $t\in [0,T]$
By taking conditional expectations with respect
to $\calF_t$ for \eqref{eq:unique2}, we see that
$ \wt U(t) - U(t) = 0$. Thus $\int_t^T
S(s-t)(\wt V - V)\ud W(s) = 0 $ for all $t\in
[0,T]$. Taking  $L^p$-means, using \cite[Theorem
3.5]{NVW07} it follows that
$$\n S(\cd-t)(\wt V - V)\n_{L^p(\Om;\ggX)}^p\eqsim_{p,X} 
\E \Big\n\int_t^T S(s-t)(\wt V(s) - V(s))\ud W(s)\Big\n^p  = 0.$$
Hence, for any $t\in [0,T]$, in $\gamma(t,T)$ we obtain the equality
$$
S(\cdot-t) \wt V(\cdot) = S(\cdot-t) V(\cdot).
$$
As before this proves that $\wt V=V$.
\end{proof}

\medskip
\noindent{\em Acknowledgment} -- The authors thank Mark Veraar for helpful comments.

\end{document}